\documentclass[10pt]{article}
\title{Modeling of a heat equation with a Dirac density}
\author{Scott W Hansen
	\thanks{Department of Mathematics, Iowa State University, Ames, IA 50010, USA 		(shansen@iastate.edu) (jesusmtz@iastate.edu). 
	Funding for this research was provided in part by the National Science 
	Foundation under award number 
	DMS-1312952.}
	 ~and ~Jose de Jesus Martinez
	 \footnotemark[1]}
\date{}

\usepackage{amsmath}
\usepackage[utf8]{inputenc}
\usepackage[english]{babel}
\usepackage{amsthm}
\usepackage{enumerate}
\usepackage{color,colortbl} 
\usepackage{amsfonts}
\usepackage{stackrel}
\usepackage{amssymb}
\usepackage{tikz}
\usepackage{graphicx}
\usepackage{fancyhdr}
\usepackage{verbatim}
\usepackage{gensymb}
\usepackage{empheq}
\usepackage[margin=1.4in]{geometry}

\newtheorem{thm}{Theorem}\numberwithin{thm}{subsection}
\numberwithin{cor}{subsection}
\newtheorem{lemma}{Lemma}\numberwithin{lemma}{subsection}
\newtheorem{prop}{Proposition}\numberwithin{prop}{subsection}
\theoremstyle{definition}

\newtheorem{rmk}{Remark}\numberwithin{rmk}{subsection}

\newcommand{\eq}[1]{\begin{align*}#1\end{align*}}
\newcommand{\eql}[2]{\begin{align}\label{#1}#2 \end{align}}
\def\<#1>{\langle#1\rangle}

\newcommand{\R}{\mathbb{R}}
\newcommand{\C}{\mathbb{C}}
\newcommand{\A}{\mathcal{A}}
\newcommand{\hi}{\mathcal{H}}
\newcommand{\W}{\mathcal{W}}
\newcommand{\sector}{{\bf S}}

\newcommand{\ep}{\epsilon}
\newcommand{\dx}{~dx}
\newcommand{\dt}{~dt}
\newcommand{\weakstar}{~\stackrel{*}{\rightharpoonup}~}
\newcommand{\oep}[1]{\omega_{\epsilon, #1}}
\newcommand{\oepm}{\omega_\ep}
\newcommand{\e}[1]{{ #1}_{\ep}}
\newcommand{\Ae}{{\A_{\ep}}}

\newcommand{\hie}{\mathcal{H}_\ep}
\newcommand{\ye}{{y_\ep}}
\newcommand{\yeo}{{y^0_\ep}}
\newcommand{\ueo}{{u^0_\ep}}
\newcommand{\veo}{{v^0_\ep}}
\newcommand{\zeo}{{z^0_\ep}}
\newcommand{\ue}{{u_\ep}}
\newcommand{\ve}{{v_\ep}}
\newcommand{\ze}{{z_\ep}}

\begin{document}
\maketitle

\begin{abstract} We consider a linear hybrid system consisting of two rods connected by a thin wall of width $2\epsilon$ and density $1/2\epsilon$. 
By passing to a limit, we obtain a system describing heat flow of two rods connected by a ``point mass'' whose dynamics are governed by a differential equation. 
We prove that solutions of the {\it epsilon} problem converge weakly to solutions of the limiting system.
\end{abstract}

\section{INTRODUCTION}
Consider a linear hybrid system consisting of two wires or rods connected by a thin wall of width $2\epsilon>0$ and density $1/2\epsilon$. 
Assume the two rods occupy the intervals $\oep{1} = (-L_1,-\ep)$ and $\oep{2}= (\ep,L_2)$, and the wall occupies the interval $\oepm = (-\ep,\ep)$.
Correspondingly, let $\ue=\ue(t,x)$, $\ve=\ve(t,x)$ and $\ze=\ze(t,x)$ denote the temperature distribution on their respective domains $\oep{1}$, $\oep{2}$, and $\oepm$.
We suppose the temperature of the rods and wall satisfy the heat equation on their respective domains with Dirichlet boundary conditions at endpoints $x=-L_1,L_2$.
The linear equation modeling heat flow of such a system is as follows:
\eql{approx sys}{
&\begin{cases}
	c_1 \rho_1\dot{\ue} - k_1\ue'' = 0, & t>0,\  x\in \oep{1}\\
	c_2 \rho_2\dot{\ve} - k_2\ve'' = 0, & t>0,\  x\in \oep{2}\\
	\frac{c}{2\ep}\dot{\ze} - k\ze'' = 0, & t>0, \ x\in \oepm\\
	\ue(t,-\ep) = \ze(t,-\ep),\ \ze(t,\ep) = \ve(t,\ep), & t>0\\
	k_1\ue'(t,-\ep) = k\ze'(t,-\ep),\ k\ze'(t,\ep) = k_2\ve'(t,\ep), & t>0\\	
	\ue(t,-L_1) = \ve(t,L_2)=0, & t>0.
\end{cases}
}
Throughout this article, $~'$ will denote spatial derivatives and $\dot{~~}$ will denote temporal derivatives. 
The parameters $c>0$ and $k>0$ in the third equation represent the specific heat and conductivity of the wall connecting the two rods.
The parameters $c_i$, $\rho_i$ and $k_i$ in \eqref{approx sys} are positive and represent the specific heat, density and thermal conductivity of the rod on the subdomain $\oep{i}$.
It will later be convenient to use the diffusivity coefficient $\alpha_i^2=k_i/c_i\rho_i$ for $i=1,2$.
The fourth equation guarantees continuity of the temperature across the interface $x=\pm\ep$ and the fifth equation represents the heat flux continuity condition at the interfaces (see \cite[Chapter 8]{Ozisik} ). 
We complete the system by adding the initial conditions 
\eql{approx init data}{
	\{\ueo(x), \veo(x), \zeo(x)\}=\{\ue(0,x), \ve(0,x), \ze(0,x)\}
}
in an appropriately defined function space at time $t=0$ so we may determine the solution of \eqref{approx sys} uniquely. 

We show in this article that with appropriate assumptions on the initial conditions, the solution $\{\ue,\ve,\ze\}$ of \eqref{approx sys} with \eqref{approx init data} converges in a weak sense to the solution of the following limiting hybrid system:
\eql{limit sys}{
&\begin{cases}
	c_1 \rho_1 \dot{u} - k_1u'' = 0, & t>0, \ x\in \omega_1:= \omega_{0,1} \\
	c_2 \rho_2 \dot{v} - k_2v'' = 0, & t>0, \ x\in \omega_2:= \omega_{0,2} \\
	c~\dot{z} = k_2 v'(t,0) - k_1 u'(t,0), & t>0\\
	u(t,0) = v(t,0) = z(t), & t>0\\	
	u(t,-L_1) = v(t,L_2)=0, & t>0,
\end{cases}
}
with initial conditions of the form
\eql{lim init data}{
	\{u^0(x), v^0(x), z^0\}=\{u(0,x), v(0,x), z(0)\}
} 
given in an appropriately defined function space at time $t=0$.
The third equation in \eqref{limit sys} states that the rate of change in temperature of the point mass is proportional to the net heat flux into the point mass.
This can be viewed as a form of Fick's law of diffusion.

Similar hybrid systems involving strings and beams with point masses have been studied in the context of controllability and stabilization theory.
See for example \cite{HZwavepoint}, \cite{LitTay}, \cite{CZbeam}, \cite{CZeulerbeam}, \cite{LMbeam}, \cite{MoRaoCo}, \cite{CoMo}, \cite{ZhaoWeiss}, \cite{GuoIvanov} and \cite{Guo}.
In particular, C. Castro showed in \cite{Cwavepoint} that a system similar to \eqref{limit sys} with strings can be obtained from a system similar to that in \eqref{approx sys} and gave a detailed spectral analysis. 


\section{WELL-POSEDNESS}
\setcounter{subsubsection}{0}

We begin by proving well-posedness of the limit problem \eqref{limit sys}.


\subsection{The limit problem}
\label{subsec: limit problem}

Given $u$, $v$ defined on $\omega_1$, $\omega_2$ and $z\in\R$, let $y=(u,v,z)^t$, where ${}^t$ denotes transposition and define
\eq{
	\hi = L^2(\omega_1)\times L^2(\omega_2)\times\R
}
equipped with the norm
\eq{
	\|y\|^2_{\hi} = c_1\rho_1\|u\|^2_{\omega_1}+c_2\rho_2\|v\|^2_{\omega_2}+c~|z|^2
}
where $\|\cdot\|_{\omega_i}$ is the usual norm in $L^2(\omega_i)$ for $i=1,2$.
Define
\eq{
	&\vartheta_{\omega_1}=\{u\in H^1(\omega_1)~|~u(-L_1)=0\}\\
	&\vartheta_{\omega_2}=\{v\in H^1(\omega_2)~|~v(L_2)=0\}\\
	&\vartheta=\{(u,v)\in \vartheta_{\omega_1}\times\vartheta_{\omega_2}~|~u(0)=v(0)\}
}
equipped with the norms
\eq{
	\|u\|^2_{\vartheta_{\omega_i}} = k_i \|u'\|^2_{L^2(\omega_i)}, \quad 
	\|(u,v)\|^2_\vartheta = \|u\|^2_{\vartheta_{\omega_1}}+\|v\|^2_{\vartheta_{\omega_2}}
}
for $i=1,2$. 
We can check that (see \cite{HZwavepoint}) $\vartheta$ is algebraically and topologically equivalent to $H^1_0(\Omega)$, however one can think of $\vartheta$ as a subspace of $\vartheta_{\omega_1}\times\vartheta_{\omega_2}$. 
The space
\eq{
	\W = \{ (u,v,z)\in\vartheta\times\R~|~u(0)=v(0)=z\}
}
is a closed subspace of $\vartheta\times\R$ with norm we may define as $\|y\|^2_{\W} = \|(u,v)\|^2_\vartheta$.
It is easy to see that the space $\W$ is densely and continuously embedded in the space $\hi$.
Define the unbounded operator $\A:D(\A)\subset\hi\rightarrow \hi$ by
\eql{limit generator}{
\A=
\begin{pmatrix} 
	\alpha^2_1d^2&0&0\\
	0& \alpha^2_2d^2&0\\
	- \frac{k_1}{c}\delta_0d & \frac{k_2}{c}\delta_0d & 0\\
\end{pmatrix}
}
where $d$ denotes the (distributional) derivative operator and $\delta_0$ denotes the Dirac delta function with mass at $x=0$, and the domain $D(\A)$ of $\A$ is given by
\eql{domain of generator}{
	D(\A)=\{y\in \W~:~ u\in H^2(\omega_1),~v\in H^2(\omega_2)\}.
}
When $D(\A)$ is endowed with the graph-norm topology
\eq{
	\|y\|^2_{D(\A)} = \|y\|^2_{\hi}+\|\A y\|^2_{\hi}
}
it becomes a Hilbert space with continuous embedding in $\hi$. 
We can therefore write the limit system \eqref{limit sys} as
\eql{limit cauchy problem}{
	\dot{y}(t) = \A y(t), \quad y(0) = y^0,\quad t>0
}
where $y^0=(u^0,v^0,z^0)$. Since $D(\A)$ is dense in $\W$ and the latter is dense in $\hi$, it follows that $\A$ is a densely defined operator. 
\begin{lemma}\label{lem: limit A is a bijection}
The operator $\A:D(\A)\rightarrow\hi$ is a bijection.
\end{lemma}
\begin{proof}
Let $\vec{f}=(f,g,h)\in \hi$ be arbitrary. 
Then the solution to $\A y=\vec{f}$ is given by
\eql{sol to Ay is f}{
	y=
	\begin{pmatrix}
		C_u(x+L_1) - F(x)\\
		C_v(x-L_2) - G(x)\\
		C_z
	\end{pmatrix}
}
where
\eql{coeffs in sol to Ay=f}{
\begin{split}
	&F(x) = \int_{-L_1}^{x}\int_{s}^{0}\alpha_1^{-2}f(r)~dr ds, \\
	&G(x) = \int_{x}^{L_2}\int_{0}^{s}\alpha_2^{-2}g(r)~dr ds,
\end{split}
\begin{split}
	&C_u = \frac{-chL_2+k_2(F(0)-G(0))}{k_2L_1+k_1L_2}\\
	&C_v = \frac{chL_1+k_1(F(0)-G(0))}{k_2L_1+k_1L_2}\\
	&C_z = -\frac{chL_1L_2+L_2k_1F(0)+L_1k_2G(0)}{k_2L_1+k_1L_2}.
\end{split}
}
Since $(u'',v'')=(\alpha_1^{-2}f, \alpha_2^{-2}g)\in L^2(\omega_1)\times L^2(\omega_2)$ it follows that $(u,v)\in H^2(\omega_1)\times H^2(\omega_2)$. 
Furthermore, one can check from \eqref{coeffs in sol to Ay=f} that $u(0)=v(0)=z$ and $u(-L_1)=v(L_2)=0$ so that $y\in D(\A)$.
Thus $\A: D(\A)\rightarrow\hi$ is surjective. 

Finally, note that the null space of $\A$ is trivial since when $\vec{f}=(0,0,0)$ we see that $y$ is the trivial solution.
Then $\A$ is injective and hence bijective. 
\end{proof}
\begin{lemma}\label{lem: limit A symmetric and dissipative}
The operator $\A:D(\A)\rightarrow\hi$ is symmetric and dissipative.
\end{lemma}
\begin{proof}
Consider $\varphi=(\mu,\nu,\zeta)\in D(\A)$. Then
\eq{
  \<\A y,\varphi>_\hi
  &=k_1u'\mu|^0_{-L_1}-k_1\< u',\mu'>_{\omega_1} 
    +k_2v'\nu|^{L_2}_{0}- k_2\< v',\nu'>_{\omega_2} 
    +(k_2v'(0)-k_1u'(0))\zeta\\
   &=-k_1u\mu'|^0_{-L_1}+k_1\< u,\mu''>_{\omega_1} 
    -k_2v\nu'|^{L_2}_{0}+ k_2\< v,\nu''>_{\omega_2} \\
   &=c_1\rho_1\<u,\alpha_1^2\mu''>_{\omega_1} 
    + c_2\rho_2\<v,\alpha_2^2\nu''>_{\omega_2} 
    +c z\left(\frac{k_2}{c}\nu'(0)-\frac{k_2}{c}\mu'(0)\right)\\
  &= \< y,\A\varphi>_\hi
}
for all $y=(u,v,z)\in D(\A)$. 
Hence $D(\A)\subset D(\A^*)$ and so $\A$ is a symmetric operator. 
In particular, when we choose $y=\varphi$ we see from the above computation that 
\eq{
	\<\A y,y>_\hi
		=-k_1\|u'\|^2_{\omega_1} -k_2\|v'\|^2_{\omega_2} 
		=-\|y\|^2_{\W}.
}
Thus we have that $\<\A y,y>_\hi\leq 0$ for any $y\in D(\A)$ as needed to show $\A$ is dissipative.
\end{proof}
\begin{lemma}\label{lem: A is self adjoint}
The operator $\A$ is closed, self-adjoint and its inverse is a compact operator in $\hi$.
\end{lemma}
\begin{proof}
As mentioned before, $\A$ is densely defined in $\hi$ and from Lemmas \ref{lem: limit A symmetric and dissipative} and \ref{lem: limit A is a bijection} we have that it is symmetric and $R(\A)=\hi$. 
It follows from Theorem 13.11 in \cite{Rudin}, that $\A$ is self-adjoint and its inverse $\A^{-1}$ is bounded in $\hi$. 
Furthermore, since the inverse is bounded, we have $0\in\rho(\A)$ and Theorem 13.9 in \cite{Rudin} implies that $\A$ is closed.

Next we claim that $K:=\A^{-1}$ is compact. 
From formulas \eqref{sol to Ay is f}-\eqref{coeffs in sol to Ay=f} we  can decompose $K=K_1+K_2$ where
\eq{
	K_1\vec{f}= 
	\begin{pmatrix}
		- F(x)\\
		- G(x)\\
		0
	\end{pmatrix}
	,\quad
	K_2\vec{f}= 
	\begin{pmatrix}
		C_u(x+L_1)\\
		C_v(x-L_2)\\
		C_z
	\end{pmatrix}.
}
Since the mappings $f\mapsto F(x)$ and $g\mapsto G(x)$ are Volterra-type operators, $K_1$ is compact. 
Since $0\in\rho(\A)$ and $K_1$ is compact, $K_2$ must be bounded. 
Since also $K_2$ has finite rank, it follows that $K_2$ is compact, and hence also $K$ is compact.
\end{proof}
\begin{prop}\label{prop: limit A c0 analytic compact}
The operator $\A$ is the infinitesimal generator of a strongly continuous $C_0$-semigroup of contractions which extends for $Re(t)>0$ to an analytic semigroup.
\end{prop}
\begin{proof}
By Lemma \ref{lem: A is self adjoint}, we have $\A$ is a closed, densely defined and self-adjoint operator and by Lemma \ref{lem: limit A symmetric and dissipative} we see that both $\A$ and $\A^*$ are dissipative. Therefore, by the  L\"{u}mer-Phillips theorem (see Luo et al \cite{LuGuMo}), we have that $\A$ generates a $C_0$-semigroup of contractions.

Furthermore, the spectrum $\sigma(\A)$ of $\A$ is contained in $(-\infty,0)$ and by the computations shown in \cite[page 55]{BePraDel}, we obtain $\|R(z,\A)\|\leq \sec(\theta/2)/|z|$
for all $z=\rho e^{i\theta}\in\C\setminus(-\infty,0]$ and $\theta\in(\pi/2,\pi)$. 
Rewriting the angle $\theta$ as $\pi/2+\delta$ where $0<\delta<\pi/2$ and letting $M=\sec(\theta/2)$ we see that 
\eq{
	\|R(z,\A)\|\leq \frac{M}{|z|} \text{ for all } z\in\sector_\theta
}
where 
\eq{
	\sector_{\theta}=\left\{z\in\C~:~ |\arg z|<\theta\right\}.
}	
Note that $\sector_{\theta}\cup\{0\}$ is contained in the resolvent set $\rho(\A)$ of $\A$. 
By Theorem 5.2. in \cite{Pazy} we have that $T$ can be extended to an analytic semigroup in a sector $\sector_{\delta}$. 
If $\text{Re}(z)>0$ then for $\delta<\pi/2$ large enough, that $z\in\sector_\delta$.
Hence $A$ generates an analytic semigroup in the right half plane $\text{Re}(z)>0$.
\end{proof}
As a consequence of Proposition \ref{prop: limit A c0 analytic compact}, given initial data $y^0\in\hi$, there exists a unique solution
\eql{limit solution space}{
	y\in C([0,\infty);\hi)
}
to the Cauchy problem \eqref{limit cauchy problem}. 
If in addition, $y^0\in D(\A)$, then $y\in C([0,\infty);D(\A))$.


\subsection{The approximate problem}
\label{subsec: approx problem}
Now consider functions $\e{u}$, $\e{v}$ and $\e{z}$ defined on $\oep{1}$, $\oep{2}$ and $\oepm$ respectively and define $\ye=(\ue,\ve,\ze)^t$.
Let 
\eq{
	\e{\hi} = L^2(\oep{1})\times L^2(\oep{2})\times L^2(\oepm).
}
equipped with the norm 
\eq{
	\|\ye\|^2_{\hie}=c_1\rho_1\|\ue\|^2_{\omega_{\ep,1}}+c_2\rho_2\|\ve\|^2_{\omega_{\ep,2}}
		+\frac{c}{2\ep}\|\ze\|^2_{\omega_\ep}
}
where $\|\cdot\|_{\omega_{\ep,i}}$ is the usual norm in $L^2(\oep{i})$ for $i=1,2$ and $\|\cdot\|_{\omega_\ep}$ is the usual norm in $L^2(\oepm)$.
Define
\eq{
	&\vartheta_{\omega_{\ep,1}}=\{\ue\in H^1(\oep{1})~|~\ue(-L_1)=0\}\\
	&\vartheta_{\omega_{\ep,2}}=\{\ve\in H^1(\oep{2})~|~\ve(L_2)=0\}
}
equipped with the norms
\eq{
	&\|\ue\|^2_{\vartheta_{\omega_{\ep,i}}} = k_i \|\ue'\|^2_{L^2(\oep{i})},
}
for $i=1,2$. 
Next, consider the following subspace of $\vartheta_{\omega_{\ep,1}}\times\vartheta_{\omega_{\ep,2}}\times H^1(\oepm)$:
\eq{
	\e\W = \{\ye\in \vartheta_{\omega_{\ep,1}}\times\vartheta_{\omega_{\ep,2}}\times H^1(\oepm)~|~\ue(-\ep)=\ze(-\ep),~\ze(\ep)=\ve(\ep)\}
}
with the norm
\eq{
	\|\ye\|^2_{\e\W}
		=k_1 \|\ue'\|^2_{\oep{1}}+k_2 \|\ve'\|^2_{\oep{2}}+k\|\ze'\|^2_{\oepm}.
}
\begin{rmk}\label{rmk: equivalence of We}
It is easy to show that the spaces $\e\W$ are uniformly equivalent to $H^1_0(\Omega)$ in the sense that there exists some constant $C>0$ such that
\eq{
	\frac{1}{C} \|\varphi\|_{\e\W} \leq \|\varphi\|_{H^1_0(\Omega)} \leq C \|\varphi\|_{\e\W}
}
where $C$ is independent of $\ep$ for all $0<\ep<\ep_0$ with finite $\ep_0$.
Furthermore, it is easy to see that the space $\e{W}$ is densely and continuously embedded in the space $\hie$.
\end{rmk}
We will also make use of the space $\hie^2=H^2(\oep{1})\times H^2(\oep{2})\times H^2(\oepm)$.
Define the unbounded operators $\Ae:D(\Ae)\subset\hie\rightarrow\hie$ by 
\eql{approx generator}{
\Ae=
\begin{pmatrix} 
	\alpha_1^2d^2&0&0\\
	0& \alpha_2^2d^2&0\\
	0 & 0 & \frac{2\ep k}{c}d^2\\
\end{pmatrix},
}
with domain $D(\Ae)$ given by 
\eq{
	D(\Ae)=\left\{\ye\in \e{\W}~:~ \ye \in \hie^2 , ~k_1\ue'(-\ep)=k\ze'(-\ep),~k\ze'(\ep)=k_2\ve'(\ep)\right\}.
}
When $D(\Ae)$ is equipped with the graph-norm topology 
\eq{
	\|\ye\|^2_{D(\Ae)} = \|\ye\|^2_{\hie} + \|\Ae\ye\|^2_{\hie},
}
it becomes a Hilbert space with continuous embedding in $\hie$. We can now rewrite system \eqref{approx sys} as a Cauchy problem: 
\eql{approx cauchy problem}{
	\dot{\ye}(t) = \Ae \ye(t), \quad \ye(0) = \yeo\in \hie, \quad t>0.
}
It is easy to see that $\Ae$ is densely defined on $\hie$.
As in Section \ref{subsec: limit problem} we have the following results.
\begin{lemma}\label{lem: approx A properties}
The operator $\Ae:D(\Ae)\rightarrow\hie$ is a bijective, dissipative, closed, self-adjoint operator with a compact inverse in $\hie$.
\end{lemma}
\begin{prop}\label{prop: approx A c0 analytic compact}
The operator $\Ae$ is the infinitesimal generator of a strongly continuous $C_0$-semigroup of contractions which extends for $Re(t)>0$ to an analytic semigroup.
\end{prop}
The fact that $\Ae$ is the infinitesimal generator of an analytic $C_0$-semigroup implies that for all $\yeo\in \hie$ there exists a unique solution 
\eql{approx solution space}{
\ye\in C([0,\infty);\hie)
}
to \eqref{approx cauchy problem}.
Moreover, if $\yeo\in D(\Ae)$, then also $\ye\in C([0,\infty);D(\Ae))$.

\section{WEAK CONVERGENCE}

The energy functional of the hybrid system \eqref{limit sys} is given by $E(t)=\|y\|^2_\hi /2$.
By taking test functions $\varphi\in C^1_0\big([0,\infty)\times\Omega\big)$, a weak form of the hybrid system \eqref{limit sys} is given by
\eql{limit weak form}{
\begin{split}
	\int\limits_{\omega_1}c_1\rho_1u^0\varphi(0,x)~dx
	&+\int\limits_{\omega_2}c_2\rho_2v^0\varphi(0,x)~dx
	+c z^0\varphi(0,0)\\
	&=-\int\limits^\infty_0\bigg\{
		\int\limits_{\omega_1}c_1\rho_1u\dot\varphi~dx
		+\int\limits_{\omega_2}c_2\rho_2v\dot\varphi~dx
		+c z\dot\varphi(t,0)\bigg\}~dt\\
	&+
	\int\limits^\infty_0\bigg\{
		\int\limits_{\omega_1}k_1u'\varphi'~dx
		+\int\limits_{\omega_2}k_2v'\varphi'~dx\bigg\}~dt.
\end{split}
}
On the other hand, the energy functional for the $\ep$-dependent problem \eqref{approx sys} is $\e{E}(t)=\|\ye\|^2_{\hie}/2$
and by taking test functions $\varphi\in C^1_0\big([0,\infty)\times\Omega\big)$, a weak form is
\eql{approx weak form}{
\begin{split}
	&\int\limits_{\oep{1}}c_1\rho_1\ueo\varphi(0,x)~dx
	+\int\limits_{\oep{2}}c_2\rho_2\veo\varphi(0,x)~dx
	+\int\limits_{\oepm}\frac{c}{2\ep}\zeo\varphi(0,x)~dx\\
	&\qquad\qquad=
	-\int\limits^\infty_0\bigg\{
		\int\limits_{\oep{1}}c_1\rho_1\ue\dot\varphi~dx
		+\int\limits_{\oep{2}}c_2\rho_2\ve\dot\varphi~dx
		+\int\limits_{\oepm}\frac{c}{2\ep}\ze\dot\varphi~dx\bigg\}~dt\\
	&\qquad\qquad\quad+
	\int\limits^\infty_0\bigg\{
		\int\limits_{\oep{1}}k_1\ue'\varphi'~dx
		+\int\limits_{\oep{2}}k_2\ve'\varphi'~dx
		+\int\limits_{\oepm}k\ze'\varphi'~dx\bigg\}~dt.
\end{split}
}
We give sufficient conditions such that we may pass to the limit in \eqref{approx weak form} to consequently obtain \eqref{limit weak form}.
Assume that $\yeo\in D(\Ae)$ and furthermore, there exists $M_1>0$ such that 
\eql{cond: initial data unif bound 1}{
	\|\yeo\|_{\hie}\leq M_1,
}
for all $\ep>0$.
Then we obtain the following result.
\begin{lemma}\label{lem: approx energy bounded}
The energy of system \eqref{approx sys} is (uniformly) bounded by the initial energy in the sense that there exists constant $C>0$ such that $\e{E}(t)\leq C$ for all $\ep>0$ whenever the initial data $\yeo$ satisfies \eqref{cond: initial data unif bound 1}.
\end{lemma}
\begin{proof} Note that if $\yeo\in D(\Ae)$, then $\ye\in C\big([0,\infty);D(\Ae)\big)$ and the energy satisfies
\eql{energy decay computations}{
\begin{split}
	\dot{E}_{\ep}(t) 
	&=k_1\ue'\ue\big|_{-L_1}^{-\ep}+k_2\ve'\ve\big|_{\ep}^{L_2}+k\ze'\ze\big|_{-\ep}^{\ep}\\
	&\qquad
		-\int\limits_{\oep{1}} k_1|\ue'|^2~dx
		-\int\limits_{\oep{2}} k_2|\ve'|^2~dx
		-\int\limits_{\oepm} k|\ze'|^2~dx\\
	&=-\|\ye\|^2_{\e{\W}},
\end{split}
}
which implies $\dot{E}_{\ep}(t)\leq0$, and thus $E_\ep(t)\leq E_\ep(0)$.
Hence by density, there exists $C=M_1^2/2$ for which $E_\ep(t)\leq C$ for all $t>0$ and initial data satisfying \eqref{cond: initial data unif bound 1}.
Consequently, we find that $\ye\in L^\infty\big([0,\infty);\hie\big)$ for all $\ep>0$.
\end{proof}
Now assume there exists $M_2>0$ such that 
\eql{cond: initial data unif bound 2}{
	\|\yeo\|_{\e\W}\leq M_2,
}
for all $\ep>0$. 
Then we obtain the following result.
\begin{lemma}\label{lem: bndd space}
Assuming condition \eqref{cond: initial data unif bound 2} holds, the sequence solutions $\{\ye\}_{\ep>0}$ to problem \eqref{approx sys} is uniformly bounded in $L^\infty\big([0,\infty);H^1_0(\Omega)\big)$.
\end{lemma}
\begin{proof}

Since $\yeo\in D(\Ae)$ we have that $\ye\in C\big([0,\infty);D(\Ae)\big)$ and 
\eq{
	\frac{d}{dt}\frac{1}{2}\|\ye\|^2_{\e\W}
	&= k_1\dot\ue\ue'\big|^{-\ep}_{-L_1}
		+k_2\dot\ve\ve'\big|_{\ep}^{L_2}
		+k\dot\ze\ze'\big|_{-\ep}^{\ep}\\
	&\qquad-k_1\<\dot\ue,\ue''>_{\ep,1}~dx-k_2\<\dot\ve\ve''>_{\ep,2}-k\<\dot\ze\ze''>_{\ep}\\
	&=-\|\ye'\|^2_{\e\W}.
}
This shows that the sequence $\{\|\ye(t)\|^2_{\e\W}\}_{t>0}$ is monotone decreasing in $t$ and thus by density, $\|\ye(t)\|_{\e\W}\leq \|\yeo\|_{\e\W}\leq M_2$ for all $t>0$ as needed to show $\ye\in L^\infty(0,\infty;\e{\W})$. 
From Remark \ref{rmk: equivalence of We} we see that the spaces $\e\W$ and $H^1_0(\Omega)$ are equivalent and thus there exists $K>0$ independent of $\ep$ such that 
\eq{
	\|\ye\|_{L^\infty([0,\infty);H^1_0(\Omega))}\leq K
}
for all $\ep>0$, as needed to show solutions to the $\ep$-dependent problem are uniformly bounded in $\mathcal{Y}$.
\end{proof}
Next, it is natural to assume the initial data convergences in a weak sense in $\hie$. 
It is easy to see that $\ueo$ being a sequence in $L^2(\oep{1})$ implies that $\chi_{\oep{1}}\ueo$ is a sequence in $L^2(\omega_1)$. 
Likewise, $\chi_{\oep{2}}\veo\in L^2(\omega_1)$.
Then we will assume that 
\eql{weak convergence of initial data}{
\begin{cases}
	&\chi_{\oep{1}}\ueo \rightharpoonup u_0 \text{ weakly in } L^2(\omega_1) \text{ as } \ep \rightarrow 0\\
	&\chi_{\oep{2}}\veo \rightharpoonup v_0 \text{ weakly in } L^2(\omega_2) \text{ as } \ep \rightarrow 0\\
	&\frac{1}{2\ep}\int_{\oepm} \zeo\dx \rightarrow z_0 \text{ in } \R \text{ as } \ep \rightarrow 0.
\end{cases}
}
The following theorem is our main result.
\begin{thm}\label{thm: weak conv}
Let $\{\ye\}_{\ep>0}$ be the sequence of solutions to the $\ep-$dependent problem \eqref{approx sys} with initial data $\yeo$.
Assuming \eqref{cond: initial data unif bound 1}, \eqref{cond: initial data unif bound 2} and \eqref{weak convergence of initial data}, the family $\{\ye\}_{\ep>0}$ of solutions to \eqref{approx sys} problem satisfies
\eq{
	\ye \weakstar y \text{ in } L^\infty\big([0,\infty;H^1_0(\Omega)\big)
}
as $\ep\rightarrow0$ where $y$ is the weak solution to the limit problem \eqref{limit sys} with initial $y^0\in\hi$.
\end{thm} 
\begin{proof}
From Lemmas \ref{lem: approx energy bounded} and \ref{lem: bndd space}, the initial energy provides a uniform bound for the solutions $\ye\in L^\infty\big([0,\infty;H^1_0(\Omega)\big)$. 
We can then extract a subsequence of solutions (which is still denoted by the index $\ep$) such that
\begin{align}
	\label{u weak convergence}
	&\chi_{\oep{1}}\ue \weakstar u \text{ in } 
		L^\infty(0,\infty;L^2(\omega_1))\cap L^\infty(0,\infty;\vartheta_{\omega_1})\\
			\label{v weak convergence}
	&\chi_{\oep{2}}\ve \weakstar v \text{ in }
		L^\infty(0,\infty;L^2(\omega_2))\cap L^\infty(0,\infty;\vartheta_{\omega_2}).
\end{align}
Next, observe that $g_{\ep}(t):=\frac{1}{2\ep}\<1,\ze(t)>_{\ep}$ defines a function on $[0,\infty)$.
Applying Holder's inequality we see that $|g_{\ep}(t)|\leq \|\ze\|_{\ep} /\sqrt{2\ep}$.
By condition \eqref{cond: initial data unif bound 1} we have that $\{g_{\ep}\}_{\ep>0}$ is a uniformly bounded sequence in $L^\infty(0,\infty)$. 
Invoking the Banach-Alaoglu Theorem we can extract a subsequence of $z_{\ep}$ (still denoted with the index $\ep$) and find $z \in L^\infty(0,\infty)$ such that 
\eql{convergence of g}{
	g_{\ep} \weakstar z \text{ in } L^\infty(0,\infty)
}
as $\ep\rightarrow 0$.
We now pass to the limit in each of the nine terms in the characterization \eqref{approx weak form} of weak solutions of the $\ep$-problem with $\varphi\in C^1_0([0,\infty)\times\Omega)$.
Since $\varphi(0,\cdot)\in C^1_0(\Omega)$ it follows from assumption \eqref{weak convergence of initial data} on the initial data $\yeo$ that 
\eq{
	&\int_{\oep{1}}\ueo\varphi(0,x)~dx=\int_{\omega_1}\chi_{\oep{1}}\ueo\varphi(0,x)~dx
		\rightarrow \int_{\omega_1}u^0\varphi(0,x)~dx,\\
	&\int_{\oep{2}}\veo\varphi(0,x)~dx=\int_{\omega_2}\chi_{\oep{2}}\veo\varphi(0,x)~dx
		\rightarrow \int_{\omega_2}v^0\varphi(0,x)~dx.
}
Next we claim that $\int_{\oepm}\frac{1}{2\ep}\zeo\varphi(0,x)~dx\rightarrow z^0\varphi(0,0)$. 
By adding and subtracting the term $z^0\varphi(0,x)$ and applying the triangle inequality we find that
\eq{
	\left|\int_{\oepm}\frac{1}{2\ep}\zeo\varphi(0,x)~dx-z^0\varphi(0,0)\right|
	&\leq\max_{x\in\oepm}|\varphi(0,x)|\left|\frac{1}{2\ep} \int_{\oepm}\zeo-z^0~dx\right|
		+|z^0|\frac{1}{2\ep} \int_{\oepm}|\varphi(0,x)-\varphi(0,0)|~dx.
}
The first term in the right hand side above tends to zero from \eqref{weak convergence of initial data} and the boundedness of $\varphi(0,\cdot)$ on $\oepm$.
The second term tends to zero as well by the continuity of $\varphi(0,x)$.

From \eqref{u weak convergence} and \eqref{v weak convergence} we have that in particular for $\dot\varphi\in C\big([0,\infty)\times \Omega\big)$ that
\eq{
	\int^\infty_0\int_{\oep{1}} \ue\dot\varphi\dx \dt\rightarrow  \int^\infty_0\int_{\omega_1} u\dot\varphi(t,x)\dx \dt,\\
	\int^\infty_0\int_{\oep{2}} \ve\dot\varphi\dx \dt\rightarrow  \int^\infty_0\int_{\omega_2} v\dot\varphi(t,x)\dx \dt.
}

Next we want to show that 
\eql{sixth term convergence}{
	\int^\infty_0\frac{1}{2\ep}\int_{\oepm}\ze(t,x)\dot\varphi(t,x)\dx \dt 
	\rightarrow \int^\infty_0z(t)\dot\varphi(t,0) \dt.
}
Observe that
\eq{
	\Bigg| \int^\infty_0\frac{1}{2\ep}\int_{\oepm}&\ze(t,x)\dot\varphi(t,x)\dx \dt -  \int^\infty_0z(t)\dot\varphi(t,0)\dx \dt\Bigg|\\
	&=
	\left| \int^\infty_0\frac{1}{2\ep} \int_{\oepm} \ze\dot\varphi(t,x) - z\dot\varphi(t,0)\dx \dt\right|\\
	&\leq
	\left| \int^\infty_0\frac{1}{2\ep}  \int_{\oepm}\ze(\dot\varphi(t,x) -\dot\varphi(t,0))\dx \dt\right|
	+\left| \int^\infty_0\frac{1}{2\ep} \int_{\oepm}(\ze- z)\dot\varphi(t,0)\dx \dt\right|.
}
The last term tends to zero by \eqref{convergence of g}. 
Regarding the first term, observe that applying H\"older's inequality we have
\eq{
	\left| \int^\infty_0\frac{1}{2\ep}  \int_{\oepm}\ze(\dot\varphi(t,x) -\dot\varphi(t,0))\dx \dt\right|
	&\leq\int^\infty_0\frac{1}{2\ep}  \int_{\oepm}\left| \ze\right|\left|\dot\varphi(t,x) -\dot\varphi(t,0)\right|\dx \dt\\
	&\leq \int^\infty_0\frac{1}{\sqrt{2\ep}}\|\ze\|_\ep\frac{1}{\sqrt{2\ep}}\sqrt{\int_{\oepm}\left|\dot\varphi(t,x) -\dot\varphi(t,0)\right|\dx} \dt\\
	&\leq \int^\infty_0\frac{1}{\sqrt{2\ep}}\|\ze\|_\ep\sqrt{\frac{1}{2\ep}\int_{\oepm}\left|\dot\varphi(t,x) -\dot\varphi(t,0)\right|\dx} \dt.
}
Note that $\|\ze\|_\ep/\sqrt{2\ep}$ is bounded by condition \eqref{cond: initial data unif bound 1} and by the continuity of $\dot\varphi$, we have that the above tends to zero as $\ep\rightarrow0$.
Since $\varphi'\in C([0,\infty)\times \Omega)$ we have from \eqref{u weak convergence} and \eqref{v weak convergence} that 
\eq{
	\int^\infty_0\int_{\oep{1}} \ue'\varphi'\dx \dt
	=\int^\infty_0\int_{\omega_1} \chi_{\oep{1}}\ue'\varphi'\dx \dt
	\rightarrow  \int^\infty_0\int_{\omega_1} u'\varphi'\dx \dt\\
	\int^\infty_0\int_{\oep{2}} \ve'\varphi'\dx \dt
	=\int^\infty_0\int_{\omega_2} \chi_{\oep{2}}\ve'\varphi'\dx \dt
	\rightarrow  \int^\infty_0\int_{\omega_2} v'\varphi'\dx \dt.
}
Finally, the term $\int^\infty_0\int_{\oepm}k\ze'\varphi'~dx dt$ tends to zero as $\ep\rightarrow0$ since $\ze(t,\cdot)\in H^1(\oepm)$ and $\varphi'\in C([0,\infty)\times \Omega)$.

From the above discussion we now have that there exists a  convergent subsequence of $\ye$ in the weak star sense, whose limit satisfies the equation \eqref{limit weak form}. 
Since the limiting system has a unique weak solution, it follows that the convergence holds for the whole sequence $\{\ye\}_{\ep>0}$. 
Therefore, we have shown that the limiting system \eqref{limit sys} can be approximated with the sequence of $\ep$ dependent problems \eqref{approx sys} as needed. 
\end{proof}



\end{document}